\documentclass{article}
\usepackage{verbatim}
\usepackage[utf8]{inputenc}
\usepackage{amsmath, amsfonts, amsthm, amssymb}
\usepackage{latexsym}
\usepackage{graphicx}
\usepackage{cancel}
\usepackage[utf8]{inputenc}
\usepackage{enumerate}
\usepackage{mathtools}
\usepackage{url}

\newtheorem{theorem}{Theorem}

\newtheorem{lemma}[theorem]{Lemma}

\theoremstyle{definition}
\newtheorem{definition}{Definition}

\theoremstyle{remark}
\newtheorem{remark}{Remark}

\title{Implications between Induction Principles for $\mathbb{N}$ in Peano Arithmetic}
\author{João Alves Silva Júnior}
\date{\today}

\begin{document}

\maketitle

\begin{abstract}
In introductory books about natural numbers, a common kind of assertion -- often left as exercise to the reader -- is that certain forms of induction on $\mathbb{N}$ (regular/ordinary, complete/strong) are equivalent one to each other and to the well-ordering principle. This means that if P1 and P2 are two of these principles, then, under all the other usually adopted postulates for the set of natural numbers (e.g., Peano axioms other than the induction axiom), P1 implies P2 and vice-versa. In this paper, we shows that, for a reasonable formalization, based on Peano arithmetic, some of the alleged implications between these principles hold only if an additional, independent condition is assumed, namely: every nonzero natural number is a successor. This condition is a consequence of the regular induction principle, but not of other induction principles. So, it is necessary to review all currently accepted implications between induction principles and similar statements in $\mathbb{N}$. From a list of 9 properties, usually considered induction principles (or equivalent to), we identify all valid and ``almost valid'' implications between them, as well as disprove all invalid implications by counterexamples, which consist of structures similar to Peano models, but not necessarily satisfying the Peano's induction axiom, and having an order relation adequately subordinated to the zero and the successors. Final remarks are made about what can and what cannot be proved if we weaken the assumptions about the order relation. The aim of this paper is to shed light on the need for more caution when alleging equivalence between induction principles for $\mathbb{N}$.
\end{abstract}

\section{Introduction.}

An exercise in~\cite[Chapter~4, Problem 2]{RS14} asks:

\begin{quote}
``(a) Using the principle of mathematical induction, prove the principle of complete mathematical induction. (b) Conversely, show that the principle of mathematical induction logically follows from the principle of complete mathematical induction.''
\end{quote}
Earlier in the chapter, the author, Raymond Smullyan, defines these principles:
\begin{quote}
``The \emph{principle of mathematical induction} is that if a property holds for the number $0$ and if it never holds for any number $n$ without holding for $n + 1$, then it must hold for all natural numbers. (\dots) A variant of the principle of mathematical induction is the \emph{principle of complete mathematical induction}, which is this: Suppose a property $P$ of natural numbers is such that for every natural number $n$, if $P$ holds for all natural numbers less than $n$, then $P$ holds for $n$. Conclusion: $P$ holds for all natural numbers.''
\end{quote}
The above mentioned principle of mathematical induction (not the complete) is also known as \emph{regular induction principle}~\cite{OMC} or \emph{ordinary induction principle}~\cite[p.~17]{LNC08}.

At this point, it is important to clarify the goal of the quoted exercise from \cite{RS14}. If we interpret it as actually relating specific principles about the set of natural numbers, then both the alleged implications are trivially true, because that principles are true assertions about $\mathbb{N}$. But this is not the real purpose of that exercise. Instead, its intended objective is to prove that the proposed implications are still valid if we replace the known, concrete structure $(\mathbb{N}, 0, S, <)$, where $S :n\mapsto n+1$ is the successor function and $<$ is the usual strict order on $\mathbb{N}$, by an abstract one, $(X, 0_X, S_X, <_X)$, that satisfies all the axioms of $\mathbb{N}$ (after adaptation for $X$), except possibly the induction axiom, which is the regular induction principle. We pretend stay as close as possible to the Peano axioms, which form the most commonly utilized axiomatic framework to the arithmetic, but since the complete/strong induction principles use an ordering, we have to consider as well some other axioms that say what is an order relation and how it must interact to the zero and the successor function with the view to form a ``$\mathbb{N}$-like structure'', but without induction.

In the sequel, we use the notations
\begin{equation}\label{E:X<p}
  X_{<p} = \{x \in X \mid x<_X p\},
\end{equation}
\begin{equation}\label{E:SX[Y]}
  S_X[Y] = \{S_X(y) \mid y \in Y\}.
\end{equation}

\begin{definition}\label{D:PreIndPeanoModel}
A \emph{pre-inductive Peano model} is a set $X$ equipped with an element $0_X$, called the \emph{zero} of $X$, a function $S_X: X \to X$, called the \emph{successor function} of $X$, and a binary relation $<_X$ on $X$, called the \emph{order} of $X$, such that
\begin{enumerate}[(a)]
  \item \label{Itm:SXinj} $S_X$ is injective, i.e., $\forall x,x' \in X, \ x \neq x' \, \Rightarrow \, S_X(x) \neq S_X(x')$.
  \item \label{Itm:0Xnotsuc} $0_X$ is not a successor, i.e., $0_X \notin S_X[X]$.
  \item \label{Itm:<Xlinord} $<_X$ is a strict order relation, i.e.,
  \begin{enumerate}[(i)]
    \item $<_X$ is irreflexive: $\forall x \in X, \ x \not<_X x$.
    \item $<_X$ is transitive: $\forall x,y,z \in X, \ ((x <_X y)\wedge(y <_X z)) \, \Rightarrow \, x <_X z$.
  \end{enumerate}
  \item \label{Itm:<0Srelation} $<_X$ is subordinated to $0_X$ and $S_X$, in the following manner:
  \begin{enumerate}[(i)]
    \item $X_{<0_X} = \varnothing$.\label{Itm:X<0Xempty}
    \item For all $p \in X$, $X_{<S_X(p)} = \{p\} \cup X_{< p}$\label{Itm:X<SXp}.
  \end{enumerate}
\end{enumerate}
\end{definition}

We regard Definition~\ref{D:PreIndPeanoModel} an adequate set of assumptions under what it makes sense to evaluate wether or not the regular induction principle implies the complete induction principle and/or vice-versa.

\begin{definition}\label{D:MRIP,MCIP}
Let $\mathcal{P} = (X,0_X,S_X,<_X)$ be a pre-inductive Peano model. We say that $\mathcal{P}$ satisfies the
\begin{enumerate}[(a)]
  \item \emph{regular induction principle} if
  \[ \forall Y \subseteq X, \ ((0_X \in Y)\wedge(S_X[Y] \subseteq Y)) \, \Rightarrow \, Y = X. \]
  \item \emph{complete induction principle} if
  \[  \forall Y \subseteq X, \ (\forall p \in X, \ X_{< p} \subseteq Y \, \Rightarrow \, p \in X) \, \Rightarrow Y = X. \]
\end{enumerate}
\end{definition}

The set $Y$ in Definition~\ref{D:MRIP,MCIP} plays the role of the property $P$ in the statement of Problem 2 from~\cite[Chapter~4]{RS14}: $Y = \{x \in X \mid P(x)\}$. So, part (a) of Definition~\ref{D:MRIP,MCIP} is equivalent to ``if a property holds for the number $0$ and if it never holds for any number $n$ without holding for $n + 1$, then it must hold for all natural numbers''. And part (b) is equivalent to ``Suppose a property $P$ of natural numbers is such that for every natural number $n$, if $P$ holds for all natural numbers less than $n$, then $P$ holds for $n$. Conclusion: $P$ holds for all natural numbers.''.

According to Definition~\ref{D:MRIP,MCIP}, the claim that the complete induction principle implies the regular induction principle means that every pre-inductive Peano model that satisfies the first of theses principles also satisfies the second one.

The adopted nomenclature in texts about induction principles on $\mathbb{N}$ is not always the same. There are two induction principles that are often called ``complete induction principle'' or ``strong induction principle'', but some books or papers approach only one of these, while other sources mention only the other principle, and frequently the names of them are switched when one moves from a source to another. When both principles are considered in the same work, they are usually treated as logically equivalent. In this paper, we prove that these forms of induction are not equivalent, so we need to distinguish them carefully. So far, we have defined only the regular (ordinary) and the complete induction principles. The strong induction principle is defined in the next section.

In one of the steps of the proposed solution for the part (b) of that exercise, Smullyan says:
\begin{quote}
``For the case of $n = 0$, we are already given that $P$ holds for $0$. And so suppose $n \neq 0$. Then $n=m+1$ for the number $m=n-1$.''
\end{quote}
This argument presumes that the set of natural numbers is formed \emph{only} by $0$ and the successors, that is,
\begin{equation}\label{E:N=|0|US[N]}
\mathbb{N} = \{0\} \cup \{ S(n) \mid n \in \mathbb{N} \}.
\end{equation}
This is not usually considered an axiom for $\mathbb{N}$, but a consequence of the regular induction principle, which cannot be used there, to avoid circularity. Can we deduce the presumed equality, \eqref{E:N=|0|US[N]}, from the complete induction principle and the Peano axioms other than the induction axiom? The answer is no, as we show below. Worse yet, we prove that the suggested implication, `complete induction' $\Rightarrow$ `regular induction', does not hold in general.

We want to construct a pre-inductive Peano model that satisfies the complete induction principle but does not satisfy the regular induction principle. \label{Pg:Model1Begin} It suffices to take $X$ as being the ordinal number $\omega + \omega$, accompanied with its usual order, its zero ($\varnothing$) and its sucessor function ($S(x) = x \cup \{x\}$). Alternatively, consider the sets
\begin{align}
  A = \{ a_n \mid n \in \mathbb{N} \} \subset \mathbb{Q}, \qquad &\text{where } a_n = -1 + \sum_{k = 0}^{n}\frac{1}{2^{k}}, \label{E:CInotimpRI:A} \\
  B = \{ b_n \mid n \in \mathbb{N} \} \subset \mathbb{Q}, \qquad &\text{where } b_n = \sum_{k = 0}^{n}\frac{1}{2^{k}},\label{E:CInotimpRI:B}
\end{align}
\begin{equation}\label{E:CInotimpRI:X}
X = A \cup B,
\end{equation}
and the function
\begin{align}
  S_X: X &\to X \notag \\
      x &\mapsto S_X(x) = \begin{cases}
                          a_{n+1} & \mbox{if } x = a_n, \\
                          b_{n+1} & \mbox{if } x = b_n
                        \end{cases} \label{E:CInotimpRI:S}
\end{align}
(see Fig.~\ref{F:omegaplusomega}). Let $<_X$ be the usual order of $\mathbb{Q}$ restricted to $X$ and consider $0_X = 0$, the natural number zero. It is straightfoward to verify that the structure $\mathcal{P} = (X,0_X,S_X,<_X)$ is a pre-inductive Peano model.\label{Pg:Model1End}

\begin{figure}[ht]
  \centering
  \includegraphics[width=12cm]{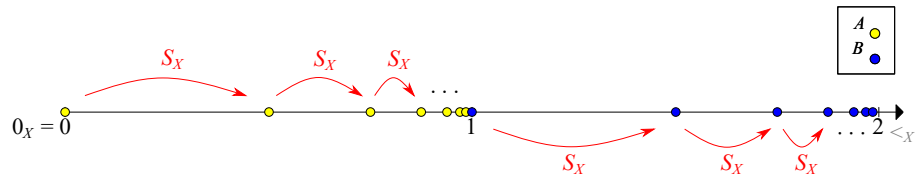}
  \caption{A counterexample for the implication `complete induction' $\Rightarrow$ `regular induction'. The pictured points represent elements of $X = A \cup B$, where $A$ and $B$ are defined by~\eqref{E:CInotimpRI:A}  and~\eqref{E:CInotimpRI:B}.}\label{F:omegaplusomega}
\end{figure}

Although we are utilizing $\mathbb{N}$ to construct a model with the intention of studying logical relations between properties of $\mathbb{N}$, there is no circularity here, because the set of natural numbers can be constructed directly from sets, independently of our model or the conclusions we want to draw from it.

Now, observe that the above constructed structure $\mathcal{P}$ satisfies the complete induction principle. In fact, let $Y \subseteq X$ be arbitrary and suppose that for all $p \in Y$, if $X_{<p} \subseteq Y$, then $p \in Y$ (induction hypothesis). Since $a_0 = 0 \in X$ is minimal (i.e., $X_{< 0} = \varnothing$), we have, by vacuity, that $X_{< a_0} = X_{<0} \subseteq Y$. So, by induction hypothesis, $a_0 \in Y$. Hence, $X_{< a_1} = \{a_0\} \subseteq Y \, \therefore \, a_1 \in Y$ and then $X_{< a_2} = \{a_0, a_1\} \subseteq Y \, \therefore \, a_2 \in Y$, etc. In general, for all $n \in \mathbb{N}$, if $a_k \in Y$ for each $k \in \{0, 1, \dots, n - 1\}$, then $X_{< a_n} = \{a_0, a_1, \dots, a_{n - 1}\} \subseteq Y$; and from this, it follows, by induction hypothesis, that $a_n \in Y$. Thus, by the complete induction principle for $\mathbb{N}$ (which we may use here, without circularity), we deduce that $a_n \in Y$ for all $n \in \mathbb{N}$. Therefore, $X_{<b_0} = A \subseteq Y$. From this, by an analogous reasoning, one can prove that
\begin{align*}
  b_0 \in Y \quad &\therefore \quad X_{<b_1} = A \cup \{b_0\} \subseteq Y \\
  \therefore \quad  b_1 \in Y \quad &\therefore \quad X_{<b_2} = A \cup \{b_0, b_1\} \subseteq Y\\
  \therefore \quad b_2 \in Y \quad &\therefore \quad X_{<b_3} = A \cup \{b_0, b_1, b_2\} \subseteq Y\\
                                   &\ \vdots
\end{align*}
so that $X = A \cup B \subseteq Y \, \therefore \, Y = X$. This concludes the proof that $\mathcal{P}$ satisfies the complete induction principle.

However, this model $\mathcal{P}$ does not satisfy the regular induction principle, because the subset $A \subseteq X$ contains $0_X$ as element and is closed under $S_X$ (i.e., $S_X[A] \subseteq A$), but $A \neq X$. Then, there is some $Y$ (namely, $Y = A$) such that the implication
\[((0_X \in Y)\wedge(S_X[Y] \subseteq Y)) \, \Rightarrow \, Y = X\]
is false.

As a result, we have that, in general, the complete induction principle does not imply the regular induction principle. So, the Problem 2(b) from \cite[Chapter~4]{RS14} is, at least, inaccuratedly stated. This is not exclusively a problem of Smullyan's book. There is a plenty of other sources with the same fail, e.g.: \cite[pp.~3, 62]{BCC08}, \cite[p.~11]{HHS12}, \cite[p.~17]{LNC08}, \cite[pp.~124-125]{PE82}. Also, a similar error often occurs when trying to relate these induction principles with the well-ordering principle for $\mathbb{N}$, e.g.: \cite[Chapter~4, Problem 3]{RS14}, \cite[p.~17]{PJC98}, \cite[p.~107]{SHW17}, \cite[pp.~124-125]{PE82}.

In the two next sections, we turn our attention to determine precisely which induction principles for $\mathbb{N}$ are equivalent, which implications between them are valid, which ones are invalid (in general), and what conditions make these false implications valid.

\section{Valid and Almost Valid Implications}

Throughout this section, we consider an arbitrary pre-inductive Peano model $\mathcal{P} = (X,0_X,S_X,<_X)$.

In addition to the notations already utilized in the previous section, we define $\leq_X$ as being the binary relation given by
\begin{equation}\label{E:<=X}
\forall x,y \in X, \quad  x \leq_X y \quad \Leftrightarrow \quad ((x <_X y)\vee(x = y))
\end{equation}
and $X_{\leq p}$ as the set
\begin{equation}\label{E:X<=p}
  X_{\leq p} = \{ x \in X \mid x \leq_X p\}.
\end{equation}
So, for all $p \in X$,
\begin{equation}\label{E:|X<=p|=|X<p|U|p|}
X_{\leq p} = X_{< p} \cup \{p\}
\end{equation}
and, by Definition~\ref{D:PreIndPeanoModel}(\ref{Itm:<0Srelation})(ii),
\begin{equation}\label{E:|X<=S(p)|=|X<=p|U|S(p)|}
X_{\leq S_X(p)} = X_{\leq p} \cup \{S_X(p)\} = X_{< p} \cup \{p,\, S_X(p)\}.
\end{equation}

\begin{definition}\label{D:IndPrinciples}
We define the following properties (induction principles) for $\mathcal{P}$.
\begin{description}
  \item[RI] (Regular Induction) \emph{For all $Y \subseteq X$, if $Y$ contains $0_X$ and is closed under $S_X$, then $Y = X$.} Symbolically,
     \[\forall Y \subseteq X, \ ((0_X \in Y) \wedge (S_X[Y] \subseteq Y)) \, \Rightarrow \, Y = X.\]
  \item[CI] (Complete Induction) \emph{Given $Y \subseteq X$, suppose that every $p \in X$ is such that if all $x \in X$ with $x <_X p$ belong to $Y$, then $p$ belongs to $Y$. Therefore, $Y = X$.} In symbols:
  \[\forall Y \subseteq X, \ (\forall p \in X,\ X_{<p} \subseteq Y  \Rightarrow p \in Y) \, \Rightarrow \, Y = X.\]
    \item[SI] (Strong Induction) \emph{Given $Y \subseteq X$, suppose that $0_X \in Y$ and that every $p \in X$ is such that if all $x \in X$ with $x \leq_X p$ belong to $Y$, then $S_X(p)$ belongs to $Y$. Therefore, $Y = X$.} In symbols:
  \[\forall Y \subseteq X, \ ((0_X \in Y)\wedge(\forall p \in X,\ X_{\leq p} \subseteq Y \Rightarrow S_X(p) \in Y)) \, \Rightarrow \, Y = X.\]
  \item[WO] (Well-Ordering) \emph{Every nonempty subset of $X$ has a minimum element (with respect to $<_X$).} Symbolically,
  \[\forall Y \subseteq X, \ Y \neq \varnothing \, \Rightarrow \, (\exists m \in Y, \, \forall y \in Y,\ m \leq_X y ). \]
  \item[WFO] (Well-Foundedness of Order) \emph{Every nonempty subset of $X$ has a minimal element with respect to $<_X$. } Symbolically,
      \[\forall Y \subseteq X, \ Y \neq \varnothing \, \Rightarrow \, (\exists m \in Y, \, \forall y \in Y,\ y \not<_X m ), \]
      where $y \not<_X m$ stands for $\neg (y <_X m)$, the negation of $y <_X m$.
  \item[WFS] (Well-Foundedness of Successor) \emph{Every nonempty subset of $X$ has a minimal element with respect to $S_X$.} Symbolically,
      \[\forall Y \subseteq X, \ Y \neq \varnothing \, \Rightarrow \, (\exists m \in Y, \, \forall y \in Y,\ S_X(y) \neq m). \]
  \item[FDO] (Finite Descent of Order) \emph{Suppose a subset $Y \subseteq X$ is such that for all $y \in Y$, there exists a $y' \in Y$ such that $y' <_X y$. Therefore, $Y = \varnothing$.} In symbols:
      \[ \forall Y \subseteq X, \ (\forall y \in Y, \, \exists y' \in Y, \ y' <_X y) \, \Rightarrow \, Y = \varnothing. \]
  \item[FDS] (Finite Descent of Successor) \emph{Suppose a subset $Y \subseteq X$ is such that for all $y \in Y$, there exists a $y' \in Y$ such that $S_X(y') = y$. Therefore, $Y = \varnothing$.} In symbols:
      \[ \forall Y \subseteq X, \ (\forall y \in Y, \, \exists y' \in Y, \ S_X(y') = y) \, \Rightarrow \, Y = \varnothing. \]
\end{description}
\end{definition}

Now, given a binary relation $\square_X$ on $X$ and an element $p \in X$, let $X_{\square p}$ denote the set of all $x \in X$ such that $x \ \square_X\ p$:
\begin{equation}\label{E:Xsquarep}
 X_{\square p} = \{x \in X \mid x \ \square_X\ p \}.
\end{equation}
This notation generalizes the previously defined $X_{<p}$ and $X_{\leq p}$. As a further example, by vieweing the successor function $S_X: X \to X$ as a binary relation on $X$, with
\[ \forall x,p \in X, \qquad x \ S_X\ p \quad \Leftrightarrow \quad S_X(x) = p,\]
we have that, for all $p \in X$, the set
\[X_{S p} = \{x \in X \mid x \ S_X\ p \} = \{x \in X \mid S_X(x) = p \} \]
is empty if $p = 0_X$ and is a singleton if $p \in S_X[X]$.

\begin{remark}\label{Rem:PrinciplesForms}
Some of the principles stated in Definition~\ref{D:IndPrinciples} may be expressed as particular cases of a single statement form, as we show below.
\begin{enumerate}[(a)]
  \item \label{Itm:Rem:PrinciplesForms:RInec} If $\mathbf{RI}$ holds, then
\begin{equation}\label{E:RIsufcond}
\forall Y \subseteq X, \ (\forall p \in X,\ X_{S p} \subseteq Y \Rightarrow p \in Y) \, \Rightarrow \, Y = X.
\end{equation}
  (Note the similarity with $\mathbf{CI}$. Later, in Theorem~\ref{T:(Almost)ValidImplications}, we refer to this property as $\mathbf{IPS}$: Induction Principle of $S_X$.) In fact, by Definition~\ref{D:PreIndPeanoModel}(\ref{Itm:0Xnotsuc}), $X_{S 0_X} = \varnothing$, so that, for every set $Y$, we have the inclusion $X_{S 0_X} \subseteq Y$ and, hence, the equivalence
  \[ (0_X \in Y) \quad \Leftrightarrow \quad (X_{S 0_X} \subseteq Y \Rightarrow 0_X \in Y).\]
  On the other hand, since $S_X$ is injective (Definition~\ref{D:PreIndPeanoModel}(\ref{Itm:SXinj})), for each $p \in S_X[X]$, there is a unique $p' \in X$ such that $S_X(p') = p$. So, $X_{S p} = \{p'\}$ and, consequently,
  \begin{align*}
    (p' \in Y \Rightarrow S_X(p') \in Y) \quad &\Leftrightarrow \quad (\{p'\} \subseteq Y \Rightarrow p \in Y)\\
    &\Leftrightarrow \quad (X_{S p} \subseteq Y \Rightarrow p \in Y).
  \end{align*}
  From this, by using again the fact that $S_X$ is injective onto $S_X[X]$, it follows that
  \begin{align*}
    S_X[Y] \subseteq Y \quad &\Leftrightarrow \quad (\forall p' \in X, \ p' \in Y \Rightarrow S_X(p') \in Y)\\
    &\Leftrightarrow \quad (\forall p \in S_X[X], \ X_{S p} \subseteq Y \Rightarrow p \in Y).
  \end{align*}
  Thus,
  \begin{align*}
    \mathbf{RI} \quad &\Leftrightarrow \quad (\forall Y \subseteq X, \ ((0_X \in Y) \wedge (S_X[Y] \subseteq Y)) \, \Rightarrow \, Y = X)\\
    &\Leftrightarrow \quad \left(\forall Y \subseteq X, \ \begin{cases} (X_{S 0_X} \subseteq Y \Rightarrow 0_X \in Y) \\ \text{and} \\(\forall p \in S_X[X],\ X_{S p} \subseteq Y \Rightarrow p \in Y) \end{cases} \, \Rightarrow \, Y = X\right)\\
    &\xRightarrow[\text{(\textasteriskcentered)}]{} \quad \left(\forall Y \subseteq X, \ (\forall p \in X,\ X_{S p} \subseteq Y \Rightarrow p \in Y) \, \Rightarrow \, Y = X\right).
  \end{align*}
  \item \label{Itm:Rem:PrinciplesForms:RIsuf} If $X = \{0_X\}\cup S_X[X]$, then property \eqref{E:RIsufcond} is also a sufficient condition for $\mathbf{RI}$. To prove this, it suffices to notice that if $X = \{0_X\}\cup S_X[X]$, then the converse of the implication $(\ast)$ in part \eqref{Itm:Rem:PrinciplesForms:RInec}, above, is true.
  \item \label{Itm:Rem:PrinciplesForms:CIIPSform} $\mathbf{CI}$ has the same form as \eqref{E:RIsufcond}, with $<_X$ instead of $S_X$. So, if $X = \{0_X\}\cup S_X[X]$, then $\mathbf{RI}$ and $\mathbf{CI}$ are both of the form
      \begin{equation}\label{E:CIRIformIf0S}
      \forall Y \subseteq X, \ (\forall p \in X,\ X_{\square p} \subseteq Y \Rightarrow p \in Y) \, \Rightarrow \, Y = X,
\end{equation}
  where the symbol $\square$ marks the place for $S$ and $<$, respectively.
  \item \label{Itm:Rem:PrinciplesForms:FDWF} Given a binary relation $\square_X$ on $X$, for all $Y \subseteq X$, the statement
      \[(\forall y \in Y, \, \exists y' \in Y, \ y'\ \square_X\ y) \, \Rightarrow \, Y = \varnothing\]
      is the contrapositive of
      \[Y \neq \varnothing \, \Rightarrow \, (\exists y \in Y, \, \forall y' \in Y,\ y' \ \cancel{\square}_X\ y),\]
  where $y' \ \cancel{\square}_X\ y$ stands for $\neg (y' \ \square_X\ y )$, the negation of $y' \ \square_X\ y$. So, the property
  \begin{equation}\label{E:FDSFDOform}
  \forall Y \subseteq X, \ (\forall y \in Y, \, \exists y' \in Y, \ y'\ \square_X\ y) \, \Rightarrow \, Y = \varnothing
  \end{equation}
  is equivalent to
  \begin{equation}\label{E:WFSWFOform}
  \forall Y \subseteq X, \ Y \neq \varnothing \, \Rightarrow \, (\exists m \in Y, \, \forall y \in Y,\ y \ \cancel{\square}_X\ m).
  \end{equation}
  Both $\mathbf{FDS}$ and $\mathbf{FDO}$ are of the form \eqref{E:FDSFDOform}, with $\square \ \equiv\ S$ and $\square \ \equiv\ <$, respectively. (We use $\equiv$ for syntactical equality.) Also, $\mathbf{WFS}$ and $\mathbf{WFO}$ are of the form \eqref{E:WFSWFOform}, with $\square \ \equiv\ S$ and $\square \ \equiv\ <$, respectively. Therefore, $\mathbf{FDS}$ is equivalent to $\mathbf{WFS}$ and $\mathbf{FDO}$ is equivalent to $\mathbf{WFO}$.
  \item \label{Itm:Rem:PrinciplesForms:WOWFO} If we assume that, for all $x,y \in X$, $x \leq_X y$ if and only if $y \not<_X x$ (i.e., if we assume that $<_X$ is a \emph{linear strict order}), then $\mathbf{WO}$ becomes equivalent to $\mathbf{WFO}$.
\end{enumerate}
\end{remark}

\begin{definition}\label{D:GenInductionPrincipleWellfoundedness}
Let $\square_X$ be a binary relation on $X$.
\begin{enumerate}[(a)]
  \item The \emph{induction principle of $\square_X$} is the property
  \[ \forall Y \subseteq X, \ (\forall p \in X,\ X_{\square p} \subseteq Y  \Rightarrow p \in Y) \, \Rightarrow \, Y = X. \]
  \item We say that $\square_X$ is \emph{well-founded} if
  \[ \forall Y \subseteq X, \ Y \neq \varnothing \, \Rightarrow \, (\exists m \in Y, \, \forall y \in Y,\ y \ \cancel{\square}_X\ m).  \]
\end{enumerate}
\end{definition}

The following result is well-known \cite[Chapter 4]{RS14}.

\begin{theorem}[Generalized Induction Theorem]\label{T:GenInductionTheorem}
Let $\square_X$ be a binary relation on $X$. Then, the induction principle of $\square_X$ is true if and only if $\square_X$ is well-founded.
\end{theorem}
\begin{proof}
For all $Y \subseteq X$, by contraposition and other elementary logic rules, we have:
\begin{align*}
&(\forall x \in X - Y,\ \exists x' \in X - Y,\ x' \,\square_X\, x)\\
\Leftrightarrow \quad &(\forall x \in X,\ x \notin Y \, \Rightarrow \, (\exists x' \in X - Y, \ x' \,\square_X\, x))\\
\Leftrightarrow \quad &(\forall x \in X,\ (\forall x' \in X - Y,\ x' \, \cancel{\square}_X\, x) \,\Rightarrow\, x \in Y)\\
\Leftrightarrow \quad &(\forall x \in X,\ (\forall x' \in X,\ x' \notin Y \,\Rightarrow\ x' \, \cancel{\square}_X\, x) \,\Rightarrow\, x \in Y)\\
\Leftrightarrow \quad &(\forall x \in X,\ (\forall x' \in X,\ x' \,\square_X\, x \,\Rightarrow\, x' \in Y) \,\Rightarrow\, x \in Y)\\
\Leftrightarrow \quad &(\forall x \in X,\ X_{\square x} \subseteq Y \,\Rightarrow\, x \in Y).
\end{align*}
Also, it is clear that $Y = X \ \Leftrightarrow \ X - Y = \varnothing$. So,
\begin{multline*}
  ((\forall x \in X - Y,\ \exists x' \in X - Y,\ x' \,\square_X\, x)\,\Rightarrow\, X - Y = \varnothing) \\
  \Leftrightarrow \quad ((\forall x \in X,\ X_{\square x} \subseteq Y \,\Rightarrow\, x \in Y) \,\Rightarrow\, Y = X).
\end{multline*}
Moreover, since the correspondence $Y \mapsto X - Y$ is a bijection from the power set of $X$ to itself, we may substitute $Y$ for $X - Y$ in the first of these two equivalent sentences, above. Of course, we may change bound variables too. So, it follows that
\begin{multline*}
  (\forall Y \subseteq X, \ (\forall y \in Y,\ \exists y' \in Y,\ y' \,\square_X\, y)\,\Rightarrow\, Y = \varnothing) \\
  \Leftrightarrow \quad (\forall Y \subseteq X, \ (\forall p \in X,\ X_{\square p} \subseteq Y \,\Rightarrow\, p \in Y) \,\Rightarrow\, Y = X).
\end{multline*}
Then, by the equivalence between the statement forms \eqref{E:FDSFDOform} and \eqref{E:WFSWFOform}, commented in Remark~\ref{Rem:PrinciplesForms}(\ref{Itm:Rem:PrinciplesForms:FDWF}), it results:
\begin{multline*}
  (\forall Y \subseteq X, \ Y \neq \varnothing \, \Rightarrow \, (\exists m \in Y, \, \forall y \in Y,\ y \ \cancel{\square}_X\ m)) \\
  \Leftrightarrow \quad (\forall Y \subseteq X, \ (\forall p \in X,\ X_{\square p} \subseteq Y  \Rightarrow p \in Y) \, \Rightarrow \, Y = X). \qedhere
\end{multline*}
\end{proof}

We say that $x \in X$ is \emph{minimal} with respect to a binary relation $\square_X$ on $X$ if there is no $x' \in X$ such that $x' \,\square_X\, x$. So, $\square_X$ is well-founded if and only if every nonempty subset of $X$ has a minimal element with respect to $\square_X$. Thus, in order that the induction principles of two different binary relations to be equivalent, it is sufficient that every minimal element with respect to any of these relations is also a minimal element with respect to the other. Now, we analyse what one can say about minimal elements with respect to $<_X$ and $S_X$.

\begin{lemma}\label{TL:Minimality}
Let $x \in X$ be given.
\begin{enumerate}[(a)]
  \item If $x$ is minimal with respect to $<_X$, then $x$ is minimal with respect to $S_X$.
  \item If $X = \{0_X\}\cup S_X[X]$ and $x$ is minimal with respect to $S_X$, then $x$ is minimal with respect to $<_X$.
\end{enumerate}
Hence, if $X = \{0_X\}\cup S_X[X]$, then $x$ is minimal with respect to $<_X$ if and only if $x$ is minimal with respect to $S_X$.
\end{lemma}
\begin{proof}
\noindent
\begin{enumerate}[(a)]
  \item By Definition~\ref{D:PreIndPeanoModel}(\ref{Itm:<0Srelation})(ii), for all $x' \in X$, if $S_X(x') = x$, then $x' < x$. So, if $x$ is not minimal with respect to $S_X$, then it is not minimal with respect to $<_X$. From this, the intended claim follows by contraposition.
  \item Suppose $X = \{0_X\}\cup S_X[X]$. By Definition~\ref{D:PreIndPeanoModel}(\ref{Itm:<0Srelation})(i), $0_X$ is minimal with respect to $<_X$. Thus, if $x$ is not minimal with respect to $<_X$, then $x \in S_X[X]$, i.e., $x = S_X(x')$ for some $x' \in X$, so that $x$ is not minimal with respect to $S_X$. Now, the proof is finished by contraposition. \qedhere
\end{enumerate}
\end{proof}

Later, in Theorem~\ref{T:(Almost)ValidImplications}, we use Theorem~\ref{T:GenInductionTheorem} and Lemma~\ref{TL:Minimality} to establish valid and ``almost valid'' implications between the induction principle of $S_X$ and the induction principle of $<_X$, which is $\mathbf{CI}$. Here, by ``almost valid'', we mean ``valid if $X = \{0_X\} \cup S_X[X]$''. All the properties listed in Definition~\ref{D:IndPrinciples} are studied there in Theorem~\ref{T:(Almost)ValidImplications}, but we have decided to state and prove some partial results separately, and two important pieces are the following theorems.

\begin{theorem}[$\mathbf{SI} \Rightarrow \mathbf{CI}$ and the converse is almost true]\label{T:SICI}\noindent
\begin{enumerate}[(a)]
  \item $\mathbf{SI}$ implies $\mathbf{CI}$.
  \item If $X = \{0_X\}\cup S_X[X]$, then $\mathbf{CI}$ implies $\mathbf{SI}$.
\end{enumerate}
Hence, if $X = \{0_X\}\cup S_X[X]$, then $\mathbf{SI}$ and $\mathbf{CI}$ are equivalent.
\end{theorem}
\begin{proof}
\noindent
\begin{enumerate}[(a)]
  \item By Definition~\ref{D:PreIndPeanoModel}(\ref{Itm:<0Srelation})(i), $X_{< 0_X} = \varnothing$. So, for every set $Y$, we have $0_X \in Y$ if and only if $X_{< 0_X} \subseteq Y \Rightarrow 0_X \in Y$. On the other hand, by Definition~\ref{D:PreIndPeanoModel}(\ref{Itm:<0Srelation})(ii), $X_{< S_X(p)} = X_{\leq p}$. So,
  \begin{multline*}
  \mathbf{SI} \\
  \begin{aligned}
  &\Leftrightarrow \  \left(\forall Y \subseteq X, \ \begin{cases} (0_X \in Y) \\ \text{and} \\(\forall p \in X,\ X_{\leq p} \subseteq Y \Rightarrow S_X(p) \in Y) \end{cases} \, \Rightarrow \, Y = X\right) \\
  &\Leftrightarrow \  \left(\forall Y \subseteq X, \ \begin{cases} (X_{< 0_X} \subseteq Y \Rightarrow 0_X \in Y) \\ \text{and} \\(\forall p \in X,\ X_{< S_X(p)} \subseteq Y \Rightarrow S_X(p) \in Y) \end{cases} \, \Rightarrow \, Y = X\right)\\
  &\Leftrightarrow \  \left(\forall Y \subseteq X, \ \begin{cases} (X_{< 0_X} \subseteq Y \Rightarrow 0_X \in Y) \\ \text{and} \\(\forall q \in S_X[X],\ X_{< q} \subseteq Y \Rightarrow q \in Y) \end{cases} \, \Rightarrow \, Y = X\right)\\
  &\xRightarrow[(\star)]{} \  (\forall Y \subseteq X, \ (\forall p \in X,\ X_{< p} \subseteq Y \Rightarrow p \in Y) \, \Rightarrow \, Y = X)\\
  &\Leftrightarrow \  \mathbf{CI}.
  \end{aligned}
  \end{multline*}
  \item If $X = \{0_X\}\cup S_X[X]$, then the converse of the implication ($\star$), above, is true. \qedhere
\end{enumerate}
\end{proof}

\begin{theorem}[$\mathbf{RI} \Leftrightarrow \mathbf{SI}$]\label{T:RIeqSI}
Regular induction and strong induction are equivalent.
\end{theorem}
\begin{proof}
We prove the implications $\mathbf{RI} \, \Rightarrow \, \mathbf{SI}$ and $\mathbf{SI} \, \Rightarrow \, \mathbf{RI}$ separately.
\begin{itemize}
\item $\mathbf{RI} \, \Rightarrow \, \mathbf{SI}$: Given a $Y \subseteq X$, define $\overline{Y} = \{p \in X \mid X_{\leq p} \subseteq Y\}$ and note that
\begin{align*}
  \overline{Y} = X \quad &\Leftrightarrow \quad X \subseteq \overline{Y} \quad \Leftrightarrow \quad (\forall p \in X, \ p \in \overline{Y})\\
  &\Leftrightarrow \quad (\forall p \in X, \ X_{\leq p} \subseteq Y)\\
  &\Leftrightarrow \quad (\forall p \in X, \, \forall x \in X, \ x \leq_X p \, \Rightarrow \, x \in Y)\\
  &\xLeftrightarrow[\circledast]{} \quad (\forall p \in X, \ p \in Y) \quad \Leftrightarrow \quad  X \subseteq Y \quad \Leftrightarrow \quad Y = X.
\end{align*}
(For equivalence $\circledast$, prove each implication separately, going through ($\forall p \in X, \ p \leq_X p \, \Rightarrow \, p \in Y$) in the `$\Rightarrow$' part; the converse is more evident, since $X \subseteq Y$ implies $\forall p \in X, \forall x \in X, \, x \in Y$.) Moreover, $X_{\leq 0_X} = \{0_X\}$, so that
\[ 0_X \in Y \quad \Leftrightarrow \quad X_{\leq 0_X} \subseteq Y \quad \Leftrightarrow \quad 0_X \in \overline{Y}. \]
Furthermore, for all $Y \subseteq X$ and $p \in X$,
\begin{multline*}
  (X_{\leq p} \subseteq Y \Rightarrow S(p) \in Y) \\
  \begin{aligned}
  &\Leftrightarrow \quad (X_{\leq p} \subseteq Y \Rightarrow ((X_{\leq p} \subseteq Y) \wedge (S(p) \in Y)))\\
  &\Leftrightarrow \quad (X_{\leq p} \subseteq Y \Rightarrow X_{\leq p} \cup \{S(p)\} \subseteq Y)\\
  &\Leftrightarrow \quad (X_{\leq p} \subseteq Y \Rightarrow X_{\leq S(p)} \subseteq Y) &&\text{(By~\eqref{E:|X<=S(p)|=|X<=p|U|S(p)|})}\\
  &\Leftrightarrow \quad (p \in \overline{Y} \Rightarrow S(p) \in \overline{Y}).
\end{aligned}
\end{multline*}
Now, suppose $\mathbf{RI}$. If $0_X \in Y$ and $\forall p \in X,\ X_{\leq p} \subseteq Y \Rightarrow S(p) \in Y$, then, since $0_X \in Y \,\Rightarrow \, 0_X \in \overline{Y}$ and
\[ \forall p \in X, \quad (X_{\leq p} \subseteq Y \Rightarrow S(p) \in Y) \,\Rightarrow\, (p \in \overline{Y} \Rightarrow S(p) \in \overline{Y}),\]
it follows that $0_X \in \overline{Y}$ and $\forall p \in X, \ p \in \overline{Y} \Rightarrow S(p) \in \overline{Y}$. So, by $\mathbf{RI}$, we have $\overline{Y} = X$, which is equivalent to $Y = X$. Thus, we have just proved that if $0_X \in Y$ and $\forall p \in X,\ X_{\leq p} \subseteq Y \Rightarrow S(p) \in Y$, then $Y = X$. That is, we have proved $\mathbf{SI}$, from $\mathbf{RI}$.
\item $\mathbf{RI} \, \Leftarrow \, \mathbf{SI}$: Observe that, for all $Y \subseteq X$ and all $p \in X$, it is immediate by \eqref{E:|X<=p|=|X<p|U|p|} that $X_{\leq p} \subseteq Y \, \Rightarrow \, p \in Y$. So,
\begin{align*}
  (p \in Y \,\Rightarrow\, S(p) \in Y) \quad
  &\Leftrightarrow \quad \begin{cases}
                (p \in Y \,\Rightarrow\, S(p) \in Y)\\
                \text{and}\\
                (X_{\leq p} \subseteq Y \,\Rightarrow\, p \in Y)
              \end{cases}\\
  &\Rightarrow \quad (X_{\leq p} \subseteq Y \,\Rightarrow\, S(p) \in Y).
\end{align*}
Now, suppose $\mathbf{SI}$. If $0_X \in X$ and $\forall p \in X,\ p \in Y \,\Rightarrow\, S(p) \in Y$, then, since
\[\forall p \in X, \quad (p \in Y \Rightarrow S(p) \in Y)\,\Rightarrow\, (X_{\leq p} \subseteq Y \Rightarrow S(p) \in Y),\]
we have $0_X \in X$ and $\forall p \in X,\ X_{\leq p} \subseteq Y \Rightarrow S(p) \in Y$. Thus, by $\mathbf{SI}$, it results that $Y = X$. So, we have just proved that if $0_X \in X$ and $\forall p \in X,\ p \in Y \,\Rightarrow\, S(p) \in Y$, then $Y = X$. That is, we have proved $\mathbf{RI}$, from $\mathbf{SI}$. \qedhere
\end{itemize}
\end{proof}

A more informal proof of Theorem~\ref{T:RIeqSI} is available in~\cite{OMC}.

Before we state and prove the main theorem of this section, we need one more auxiliary result.

\begin{lemma}\label{TL:RIfacts} If $\mathbf{RI}$ holds, then
\begin{enumerate}[(a)]
  \item \label{Itm:RI->0S} $X = \{0_X\} \cup S_X[X]$.
  \item \label{Itm:0<=x} For all $x \in X$, $0_X \leq_X x$.
  \item \label{Itm:Sincreasing} For all $x,y \in X$, if $x <_X y$, then $S_X(x) <_X S_X(y)$.
  \item \label{Itm:CI0S=>LO} For all $x,y \in X$, $x <_X y$ or $x = y$ or $y <_X x$.
\end{enumerate}
\end{lemma}
\begin{proof}
\noindent
\begin{enumerate}[(a)]
  \item The set $Y = \{0_X \} \cup S_X[X]$ is a subset of $X$ that contains $0_X$ as element and is clearly closed under $S_X$. Hence, by $\mathbf{RI}$, $Y = X$.
  \item It is clear that $0_X \leq_X 0_X$. Also, for all $x \in X$, if $0_X \leq_X x$, then, by Definition~\ref{D:PreIndPeanoModel}(\ref{Itm:<0Srelation})(ii), $0_X \leq_X x <_X S_X(x) \ \therefore \ 0_X \leq_X S_X(x)$. So, by $\mathbf{RI}$, $\{x \in X \mid 0_X \leq_X x\} = X$.
  \item Let $Y = \{y \in X \mid \forall x \in X, \ x <_X y \Rightarrow S_X(x) <_X S_X(y)\}$.
  By Definition~\ref{D:PreIndPeanoModel}(\ref{Itm:<0Srelation})(i), it is clear that $0_X \in Y$. On the other hand, let $y \in Y$ be arbitrary. So, for all $x \in X$, if $x <_X y$ then $S_X(x) <_X S_X(y)$ (induction hypothesis). For all $x \in X$, we have
      \begin{multline*}
        x <_X S_X(y)\\
        \begin{aligned}
        &\Rightarrow \quad (x <_X y)\vee(x = y) &&\text{(by Definition~\ref{D:PreIndPeanoModel}(\ref{Itm:<0Srelation})(ii))}\\
        &\Rightarrow \quad (S_X(x) <_X S_X(y))\vee(S_X(x) = S_X(y)) &&\text{(by ind. hypothesis)}\\
        &\Rightarrow \quad S_X(x) <_X S_X(S_X(y)). &&\text{(by Definition~\ref{D:PreIndPeanoModel}(\ref{Itm:<0Srelation})(ii))}
        \end{aligned}
      \end{multline*}
  So, $S_X(y) \in Y$. Since $y$ is arbitrary, we have just proved that $S_X[Y] \subseteq Y$. Thus, $Y$ contains $0_X$ as element and is closed under $S_X$. Then, by $\mathbf{RI}$, $Y = X$.
  \item Let us say that $x$ and $y$ are comparable if $x <_X y$ or $x = y$ or $y <_X x$. Let $Y$ be the set of all elements of $X$ that are comparable to all elements of $X$, that is, $Y = \{y \in X \mid \forall x \in X,\ (x <_X y) \vee (x = y) \vee (y <_X x) \}$. Given an arbitrary $y \in Y$, we want to prove that $S_X(y) \in Y$. For each $x \in X$, by part \eqref{Itm:RI->0S} of this Lemma, there are two cases to consider:
      \begin{itemize}
        \item $x = 0_X$. Then, by part \eqref{Itm:0<=x}, $x \leq_X S_X(y)$.
        \item $x \in S_X[X]$. So, there is an $x' \in X$ such that $x = S_X(x')$. Since $y \in  Y$, we have that $x'$ and $y$ are comparable. Then, one of the three following subcases occurs:
            \begin{itemize}
              \item $x' <_X y$. By part \eqref{Itm:Sincreasing}, this implies $x = S_X(x') <_X S_X(y)$.
              \item $x' = y$. So, $x = S_X(x') = S_X(y)$.
              \item $y <_X x'$. By part \eqref{Itm:Sincreasing}, this implies $S_X(y) <_X S_X(x') = x$.
            \end{itemize}
      \end{itemize}
      Therefore, in any case, $x$ and $S_X(y)$ are comparable. Since $x$ and $y$ are arbitrary, we have just proved that $S_X[Y] \subseteq Y$. Moreover, by part \eqref{Itm:0<=x}, we have $0_X \in Y$. Then, by $\mathbf{RI}$, it results $Y = X$. \qedhere
\end{enumerate}
\end{proof}

\begin{theorem}\label{T:(Almost)ValidImplications}
Let $\mathbf{IPS}$ denote the induction principle of $S_X$, as formulated in Definition~\ref{D:GenInductionPrincipleWellfoundedness}, i.e.,
\begin{equation}\label{E:IPS}
\forall Y \subseteq X, \ (\forall p \in X,\ X_{S p} \subseteq Y \Rightarrow p \in Y) \, \Rightarrow \, Y = X.
\end{equation}
Some logical relations between this property and that listed in Definition~\ref{D:IndPrinciples} are indicated in Fig.~\ref{F:ArrowDiagram}. Solid arrows represent necessarily valid implications, while dashed arrows represent implications that are valid provided that $X = \{0_X\} \cup S_X[X]$.
\end{theorem}

\begin{figure}[ht]
  \centering
  \includegraphics[scale=0.5]{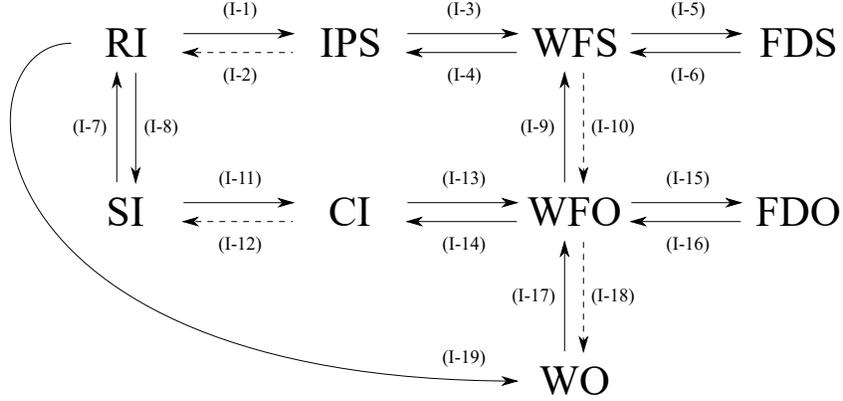}
  \caption{Implications between some forms of induction for $\mathbb{N}$: valid (solid lines) and ``almost valid'' (dashed lines). $\mathbf{IPS}$ is defined in the statement of Theorem~\ref{T:(Almost)ValidImplications}. The other principles are listed in Definition~\ref{D:IndPrinciples}.}\label{F:ArrowDiagram}
\end{figure}

\begin{proof}
\noindent
\begin{itemize}
  \item Implication (I-1): See Remark~\ref{Rem:PrinciplesForms}(\ref{Itm:Rem:PrinciplesForms:RInec}).
  \item Implication (I-2): See Remark~\ref{Rem:PrinciplesForms}(\ref{Itm:Rem:PrinciplesForms:RIsuf}).
  \item Implications (I-3), (I-4), (I-13), and (I-14): See Remark~\ref{Rem:PrinciplesForms}(\ref{Itm:Rem:PrinciplesForms:CIIPSform}) and Theorem~\ref{T:GenInductionTheorem}.
  \item Implications (I-5), (I-6), (I-15), and (I-16): See Remark~\ref{Rem:PrinciplesForms}(\ref{Itm:Rem:PrinciplesForms:FDWF}).
  \item Implications (I-7) and (I-8): See Theorem~\ref{T:RIeqSI}.
  \item Implications (I-9) and (I-10): See Lemma~\ref{TL:Minimality}.
  \item Implications (I-11) and (I-12): See Theorem~\ref{T:SICI}.
  \item Implication (I-17): Every minimum element with respect to $<_X$ is a minimal element  with respect to $<_X$, since $x \leq_X y$ implies $y \not<_X x$, by Definition~\ref{D:PreIndPeanoModel}(\ref{Itm:<Xlinord}).
  \item Implication (I-18): Suppose $X = \{0_X\} \cup S_X[X]$ and $\mathbf{WFO}$. Then, by implications (I-14), (I-12), and (I-7), we have $\mathbf{RI}$. Hence, by Lemma~\ref{TL:RIfacts}(\ref{Itm:CI0S=>LO}), $<_X$ is a linear strict order. From this, by Remark~\ref{Rem:PrinciplesForms}(\ref{Itm:Rem:PrinciplesForms:WOWFO}), it follows $\mathbf{WO}$.
  \item Implication (I-19): Suppose $\mathbf{RI}$. By implications (I-8), (I-11), and {(I-13)}, we have $\mathbf{WFO}$. On the other hand, by Lemma~\ref{TL:RIfacts}(\ref{Itm:RI->0S}), we have $X = \{0_X\} \cup S_X[X]$, and this makes implication (I-18) valid. Therefore, $\mathbf{WO}$ is true.\qedhere
\end{itemize}
\end{proof}

\section{Invalid Implications}

In this section, we disprove some implications between induction principles and related properties of pre-inductive Peano models. We begin by approaching the ``almost valid'' implications from the previous section.

\begin{theorem}\label{T:InvalidImplications}
The implications indicated by dashed arrows in Fig.~\ref{F:ArrowDiagram} are not valid in general.
\end{theorem}
\begin{proof}
We exhibit a counterexample to each one. For $\mathbf{CI} \,\Rightarrow\, \mathbf{SI}$ and $\mathbf{IPS} \,\Rightarrow\, \mathbf{RI}$, the same model defined in the Introduction (pp.~\pageref{Pg:Model1Begin}-\pageref{Pg:Model1End}) serves, because we have seen that it is a counterexample for $\mathbf{CI} \Rightarrow \mathbf{RI}$, whereas $\mathbf{RI}$ is equivalent to $\mathbf{SI}$ and $\mathbf{CI}$ implies $\mathbf{IPS}$. It remains to disprove $\mathbf{WFS} \Rightarrow \mathbf{WFO}$ and $\mathbf{WFO} \Rightarrow \mathbf{WO}$.
\begin{figure}[ht]
  \centering
  \includegraphics[scale=0.6]{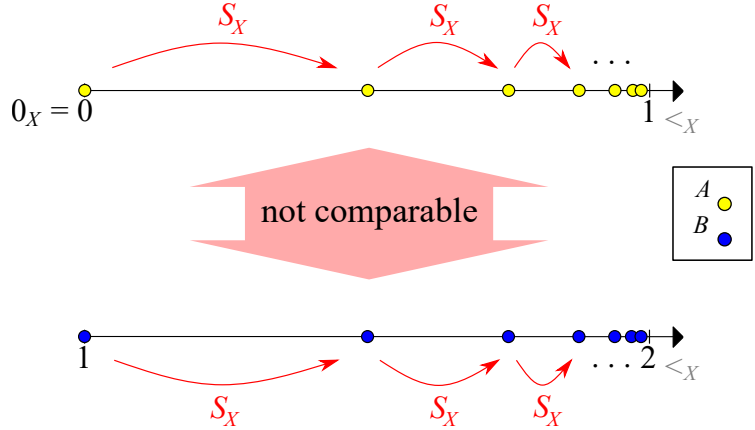}
  \caption{A counterexample for the implication $\text{`well-foundedness of order'} \Rightarrow \text{`well-ordering'}$. This model is a variant of that from Fig.~\ref{F:omegaplusomega}, with the same zero, the same successor function, but now the order is not linear. Instead, no element of $A$ is comparable with any element of $B$.}\label{F:omega1Uomega2}
\end{figure}
  \begin{itemize}
    \item $\mathbf{WFO} \Rightarrow \mathbf{WO}$: Consider again the same model constructed in the Introduction (pp.~\pageref{Pg:Model1Begin}-\pageref{Pg:Model1End}) but with its order relation $<_X$ modified so that no element of $A$ is comparable with any element of $B$ (see Fig.~\ref{F:omega1Uomega2}). It is straightforward to verify that this $\mathcal{P} = (X,0_X,S_X,<_X$) satisfies the conditions of Definition~\ref{D:PreIndPeanoModel}. Observe that inside $A$ or inside $B$, every nonempty set has a minimum element with respect to $A$ or $B$, respectively, and this relative minimum is also a relative minimal. Since we are not allowed to compare elements of $A$ with elements of $B$, the relative minimal is also a minimal with respect to the whole $X = A \cup B$. The existence of minimal element is also verified for all $Y \subseteq X$ with $A \cap Y \neq \varnothing$ and $B \cap Y \neq \varnothing$, because, in this case, there are relative minimal elements $m_1 \in A \cap Y$ and $m_2 \in B \cap Y$ and these must be absolute minimal as well. Thus, $\mathcal{P}$ satisfies $\mathbf{WFO}$. But it does not satisfy $\mathbf{WO}$, because for all $a \in A$ and all $b \in B$, the nonempty subset $\{a,b\} \subseteq X$ has no minimum element. Therefore, the implication $\mathbf{WFO} \Rightarrow \mathbf{WO}$ is false for $\mathcal{P}$.
    \item $\mathbf{WFS} \Rightarrow \mathbf{WFO}$: Consider the sets $A_n = \{ (n,k) \mid k \in \mathbb{N}\}$, with $n \in \mathbb{N}$, and $X = \bigcup_{n \in \mathbb{N}}A_n = \mathbb{N}\times \mathbb{N}$. Define the zero of $X$ by $0_X = (0,0)$. Define the successor function of $X$ by $S_X(n,k) = (n, k+1)$, for all $n,k \in \mathbb{N}$. And define the order of $X$ by
        \begin{multline*}
          (n_1, k_1) <_X (n_2, k_2) \\
          \Leftrightarrow \quad ((n_1 = n_2)\wedge(k_1 < k_2))\vee(k_1=0 < n_2 < n_1)),
        \end{multline*}
        for all $(n_1, k_1),(n_2, k_2) \in X$ (see Fig.~\ref{F:NxN}). It is straightforward to verify that this $\mathcal{P} = (X,0_X,S_X,<_X$) satisfies the conditions of Definition~\ref{D:PreIndPeanoModel}. Now, observe that $\mathbf{WFS}$ does not hold for $\mathcal{P}$ if and only if there is a nonempty $Y \subseteq X$ such that $Y \subseteq S_X[Y]$. So, in order to prove that $\mathcal{P}$ satisfies $\mathbf{WFS}$, it suffices to prove that  there is no $Y \subseteq X$ such that $\varnothing \neq Y \subseteq S_X[Y]$. But, for all $Y \subseteq X$, if $\varnothing \neq Y \subseteq S_X[Y]$, then, by applying the second projection on $Y$ and $S_X[Y]$, we have
        \[\varnothing \neq \underbrace{\{k \in \mathbb{N} \mid \exists n \in \mathbb{N},\ (n,k)\in Y\}}_{U} \subseteq \{k' + 1 \mid k' \in U\}, \]
        so that there is a nonempty $U \subseteq \mathbb{N}$ such that, for all $k \in U$, there is a $k' \in U$ such that $k = k' + 1$. This is impossible, because $\mathbb{N}$ satisfies $\mathbf{RI}$ and, hence, $\mathbf{WFS}$. So, by \emph{reductio ad absurdum}, there is no $Y \subseteq X$ such that $\varnothing \neq Y \subseteq S_X[Y]$. Equivalently, $\mathcal{P}$ satisfies $\mathbf{WFS}$. But it does not satisfy $\mathbf{WFO}$, because the subset $\{(n,0)\mid n \in \mathbb{N}\} \subseteq X$ is nonempty and has no minimal element: for all nonzero $n \in \mathbb{N}$, there is a $n' = n + 1$ such that $(n', 0) <_X (n, 0)$. In fact, the function $j \mapsto (j+1,0)$, from $\mathbb{N}$ to $X$, is strictly decreasing; $<_X$ was defined so that this happens. Therefore, the implication $\mathbf{WFS} \Rightarrow \mathbf{WFO}$ is false for $\mathcal{P}$.\qedhere
  \end{itemize}
\end{proof}

\begin{figure}[ht]
  \centering
  \includegraphics[scale=0.6]{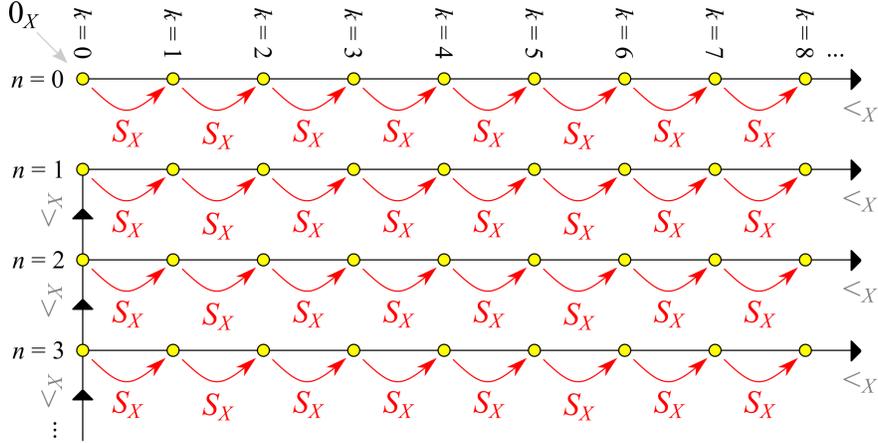}
  \caption{A counterexample for the implication `well-foundedness of successor' $\Rightarrow$ `well-foundedness of order'. $X = \{(n,k) \in \mathbb{N} \mid n,k \in \mathbb{N}\}$. The zero of this model is $0_X = (0,0)$.}\label{F:NxN}
\end{figure}

\begin{figure}[ht]
  \centering
  \includegraphics[scale=0.44]{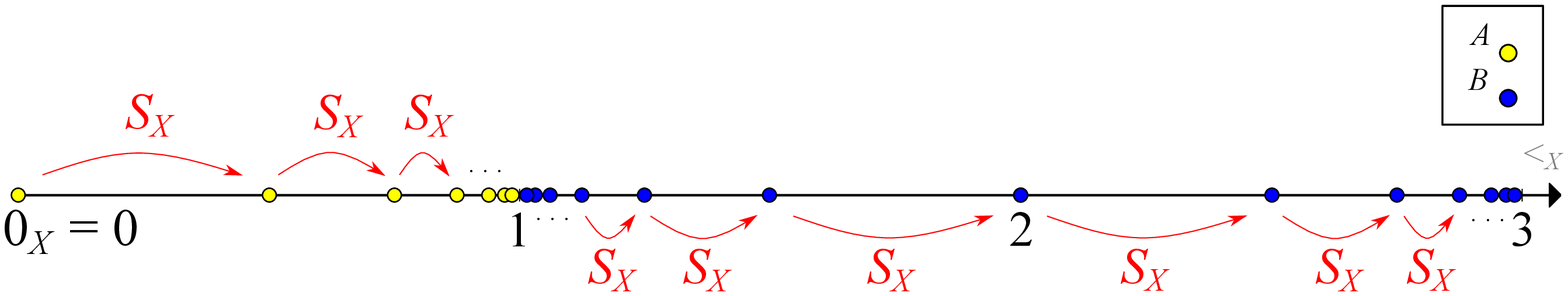}
  \caption{A pre-inductive Peano model that satisfies $X = \{0_X\} \cup S_X[X]$ but its successor function is not well-founded. Hence, it does not satisfy any induction principle listed in Definition~\ref{D:IndPrinciples} or represented in Fig.~\ref{F:ArrowDiagram}.}\label{F:omegaplusZ}
\end{figure}

\begin{remark}\label{Obs:0SAlone}
Although the property $X = \{0_X\} \cup S_X[X]$ is an indispensable condition to the validity of the ``almost valid'' implications in Fig.~\ref{F:ArrowDiagram}, it alone does not guarantee the truth of any of the induction principles shown in that diagram. To prove it, consider the sets
\begin{align}
  A = \{ a_n \mid n \in \mathbb{N} \} \subset \mathbb{Q}, \qquad &\text{where } a_n = -1 + \sum_{k = 0}^{n}\frac{1}{2^{k}}, \label{E:0SnotimpWFS:A} \\
  B = \{ b_n \mid n \in \mathbb{N} \} \subset \mathbb{Q}, \qquad &\text{where } b_n = \begin{cases}
                                                                    1 + \sum_{k = 0}^{n}1/2^{k}  & \mbox{if } n \geq 0, \\
                                                                    3 - \sum_{k = 0}^{-n}1/2^{k} & \mbox{if } n < 0,
                                                                   \end{cases}\label{E:0SnotimpWFS:B}
\end{align}
\begin{equation}\label{E:0SnotimpWFS:X}
X = A \cup B,
\end{equation}
and the function
\begin{align}
  S_X: X &\to X \notag \\
      x &\mapsto S_X(x) = \begin{cases}
                          a_{n+1} & \mbox{if } x = a_n, \\
                          b_{n+1} & \mbox{if } x = b_n
                        \end{cases} \label{E:0SnotimpWFS:S}
\end{align}
(see Fig.~\ref{F:omegaplusZ}). Observe that the structure $\mathcal{P} = (X,0_X,S_X,<_X)$ is a pre-inductive Peano model that does not satisfy $\mathbf{WFS}$: its nonempty subset $B$ is not well-founded with respect to $S_X$. Hence, this $\mathcal{P}$ does not satisfy any other induction principle or related property represented in Fig.~\ref{F:ArrowDiagram}.
\end{remark}

\section{Weakening the Assumptions}

The conditions in Definition~\ref{D:PreIndPeanoModel}(d) may not seem natural, at first glance, to some readers. Because of this, in this section, we investigate which implications in Fig.~\ref{F:ArrowDiagram} are still valid or ``almost valid'' if we weaken that clauses.

\begin{definition}\label{D:SubIndPeanoModel}
A \emph{sub-inductive Peano model} is a set $X$ equipped with an element $0_X$, called the \emph{zero} of $X$, a function $S_X: X \to X$, called the \emph{successor function} of $X$, and a binary relation $<_X$ on $X$, called the \emph{order} of $X$, such that $S_X$ is injective, $0_X$ is not in the image of $S_X$ and $<_X$ is a strict order relation.
\end{definition}

So, a pre-inductive Peano model is a sub-inductive Peano model that satisfies part~(\ref{Itm:<0Srelation}) of Definition~\ref{D:PreIndPeanoModel}.

As previously, we use the notations $S_X[X] = \{S_X(x) \mid x \in X\}$, $X_{\square p} = \{x \in X \mid x \ \square_X\  p\}$, etc. Also, we reconsider the properties defined in Definition~\ref{D:IndPrinciples}, as well that in~\eqref{E:IPS}, as being applicable to sub-inductive Peano models. In Fig.~\ref{F:ArrowDiagramMod1}, solid, smooth arrows represent valid simplications, dashed arrows represent ``almost valid'' implications (i.e., implications that are valid if $X = \{0_X\} \cup S_X[X]$), and crossed arrows represent invalid (and not ``almost valid'') implications.

\begin{figure}[ht]
  \centering
  \includegraphics[scale = 0.5]{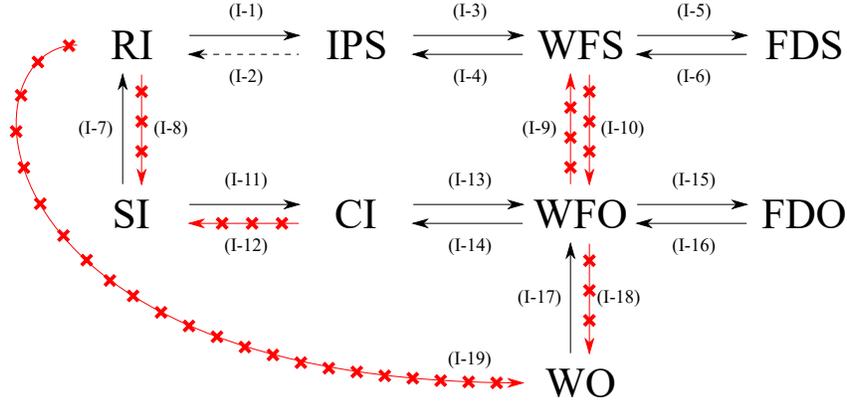}
  \caption{Valid, invalid and ``almost valid'' implications between the nine properties previously studied, in the case when $S_X: X \to X$ is injective, $0_X \in X - S_X[X]$, and $<_X$ is a strict order, without further assumptions. Here, conditions in Definition~\ref{D:PreIndPeanoModel}(\ref{Itm:<0Srelation}) are not necessarily valid.}\label{F:ArrowDiagramMod1}
\end{figure}

\begin{theorem}
For sub-inductive Peano models, the implications between the induction principles $\mathbf{RI}$, $\mathbf{SI}$, $\mathbf{CI}$, $\mathbf{IPS}$, $\mathbf{WFS}$, $\mathbf{WFO}$, $\mathbf{FDS}$, $\mathbf{FDO}$, and $\mathbf{WO}$ are as Fig.~\ref{F:ArrowDiagramMod1} shows.
\end{theorem}
\begin{proof}
The proofs of valid or ``almost valid'' implications are the same as we have made in previous sections, except for the implication (I-19): $\mathbf{SI} \Rightarrow \mathbf{CI}$. The proof of this one must be adapted as follows.

Suppose $\mathcal{S}$ satisfies $\mathbf{SI}$. Since $\mathbf{SI} \Rightarrow \mathbf{RI}$, we have that $\mathcal{S}$ satisfies $\mathbf{RI}$ and, hence, $(X,0_X, S_X)$ is isomorphic to $(\mathbb{N}, 0, S)$. So, $X$ has an order relation $\prec_X$ such that $(X, \prec_X)$ is isomorphic to $(\mathbb{N}, <)$, where $<$ is the usual order of $\mathbb{N}$. More explicitly, $\prec_X$ is recursively defined by
\[ \forall x_1, x_2 \in X, \quad x_1 \prec_X x_2 \quad \Leftrightarrow \quad (\exists x \in X, \ (x_1 \preceq_X x)\wedge(S_X(x) = x_2)),\]
where $x_1 \preceq_X x$ denotes $(x_1 \prec_X x) \vee (x_1 = x)$.

Now, observe that
\begin{equation}\label{E:X<pSubseteqXprecp}
  \forall p \in X, \quad X_{<p} \subseteq X_{\prec p}.
\end{equation}
In fact, suppose, towards a contradiction, that some $q \in X$ does not satisfy $X_{<q} \subseteq X_{\prec q}$. Then, there is a $Y = X_{\preceq q} \subseteq X$, such that $0_X \in Y$ and, for all $p \in X$, with $X_{\leq p} \subseteq Y$,
\[\{p\} \subseteq \{p\} \cup X_{< p} = X_{\leq p} \subseteq Y = X_{\preceq q},\]
so that $p \prec_X q$ or $p = q$. But, since $X_{<q} \not\subseteq X_{\prec q}$, we must have $p \neq q$: otherwise, $X_{\leq p} = X_{\leq q} \not\subseteq X_{\preceq q} = Y$. Then, there is a $Y = X_{\preceq q}  \subseteq X$ such that $0_X \in Y$ and, for all $p \in X$,
\begin{align*}
  X_{\leq p} \subseteq Y \quad &\Rightarrow \quad p \prec_X q \quad \Rightarrow \quad S_X(p) \preceq_X q\\
                               &\Rightarrow \quad S_X(p) \in X_{\preceq q} = Y,
\end{align*}
but $Y \neq X$. This contradicts $\mathbf{SI}$. So, by \emph{reductio ad absurdum}, the hypothesis that $\exists q \in X, \ X_{<q} \not\subseteq X_{\prec q}$ is false. We have proved~\eqref{E:X<pSubseteqXprecp}.

Since $(X, \prec_X)$ is isomorphic to $(\mathbb{N}, <)$, we have that $X_{\prec p}$ is finite, for all $p \in X$. Then, by \eqref{E:X<pSubseteqXprecp}, $X_{< p}$ is finite, for all $p \in X$. Therefore, $\mathcal{S}$ satisfies $\mathbf{FDO}$. From this, by $\mathbf{FDO} \Rightarrow \mathbf{WFO} \Rightarrow \mathbf{CI}$, it follows that $\mathcal{S}$ satisfies $\mathbf{CI}$.

Finally, we exhibit the counterexamples for the false (and not ``almost valid'') implications.
\begin{itemize}
  \item Implications (I-8), (I-10), and (I-19): Consider, as a counterexample, the structure $\mathcal{S} = (X, 0_X, S_X, <_X)$, where $X = \mathbb{N}$, $0_X = 0$ (the natural number zero), $S_X = S$ (the usual successor function on $\mathbb{N}$), and $<_X$ is the binary relation on $X$ defined by $x <_X y \, \Leftrightarrow \, y < x$, for all $x, y \in \mathbb{N}$, where $<$ is the usual order of $\mathbb{N}$. In other words, $<_X$ is the dual (opposite) order of $<$. See Fig.~\ref{F:NreverseOrder}.
  \item Implication (I-12): Consider, as a counterexample, the structure $\mathcal{S} = (X, 0_X, S_X, <_X)$, where $X = \{(n, \pm n) \mid n \in \mathbb{N}\} \subseteq \mathbb{N} \times \mathbb{Z}$, $0_X = (0, 0)$, $S_X: X \to X$ is given by
      \[S_X(x,y) = \begin{cases}
                    (2,-2)     & \mbox{if } y = 0, \\
                    (x,-y)     & \mbox{if } y < 0, \\
                    (x-1,1-x)  & \mbox{if } x \text{ is even and } y >0, \\
                    (x+3,-x-3) & \mbox{if } x \text{ is odd and  } y >0,
                  \end{cases}\]
      and $<_X$ is a binary relation on $X$ defined by
      \[(x_1, y_1) <_X (x_2, y_2) \quad \Leftrightarrow \quad ((y_1y_2 > 0) \wedge (x_1 < x_2)),\]
      for all $(x_1,y_1), (x_2,y_2) \in X$. See Fig.~\ref{F:Braid}.
  \item Implication (I-18): Consider, as a counterexample, the structure $\mathcal{S} = (X, 0_X, S_X, <_X)$, where $X = \mathbb{N}$, $0_X = 0$ (the natural number zero), $S_X = S$ (the usual successor function on $\mathbb{N}$), and $<_X$ is the empty relation. Alternatively, we could define $<_X$ as being any other subrelation of $<$ (the usual order or $\mathbb{N}$) such that some $x_1, x_2 \in X$ are not comparable with respect to $<_X$. See Fig.~\ref{F:NincompleteOrder}.
  \item Implication (I-9): Consider, as a counterexample, the structure $\mathcal{S} = (X, 0_X, S_X, <_X)$, where $X$, $0_X$, $S_X$ are defined as in Remark~\ref{Obs:0SAlone} and $<_X$ is constructed by modifying that order from Remark~\ref{Obs:0SAlone} so that the predecessors of $2$ (with respect to $S_X$) become incomparable, one to each other, with respect to $<_X$. See Fig.~\ref{F:omegaplusZmod}.
\end{itemize}
\end{proof}

\begin{figure}[ht]
  \centering
  \includegraphics[width = 10cm]{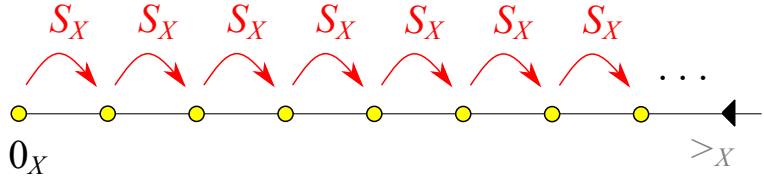}
  \caption{Counterexample for implications (I-8), (I-10), and (I-19) in Fig.~\ref{F:ArrowDiagramMod1}.}\label{F:NreverseOrder}
\end{figure}

\begin{figure}[ht]
  \centering
  \includegraphics[width = 8cm]{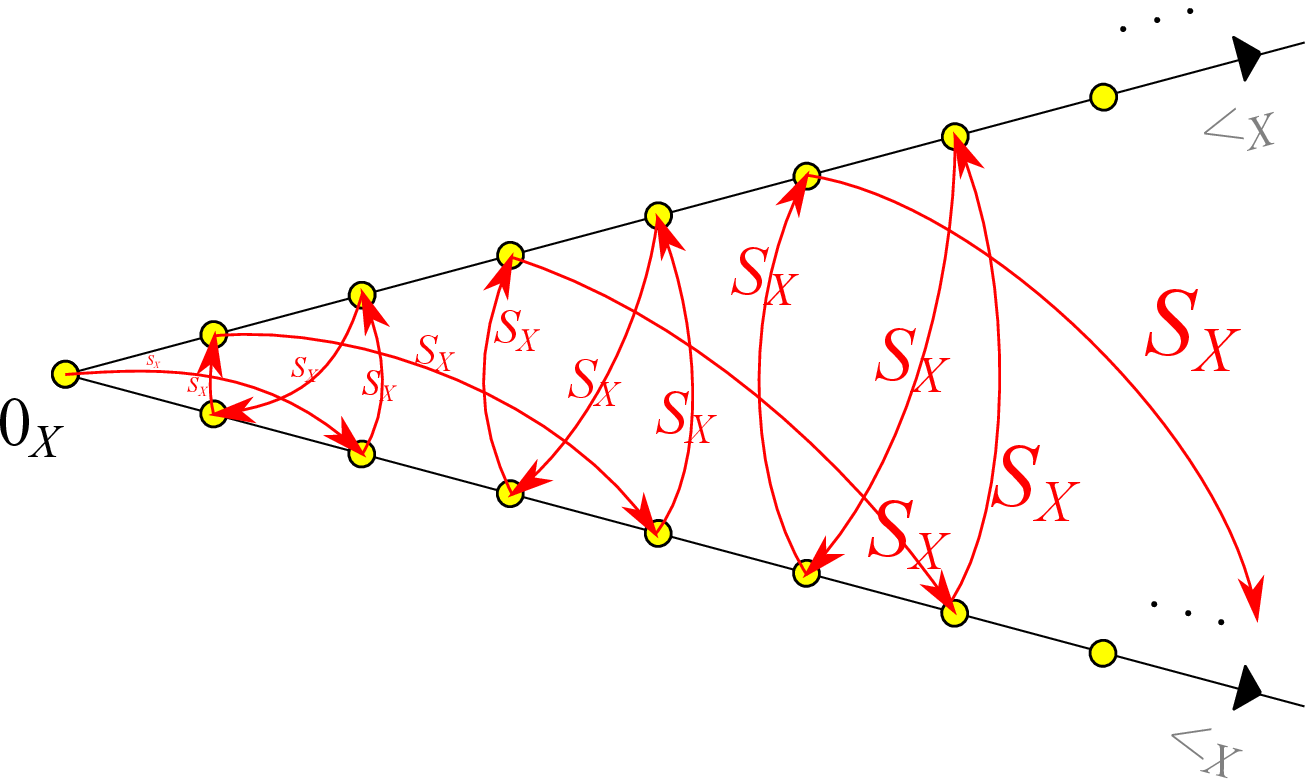}
  \caption{Counterexample for implication (I-12) in Fig.~\ref{F:ArrowDiagramMod1}.}\label{F:Braid}
\end{figure}

\begin{figure}[ht]
  \centering
  \includegraphics[width = 10cm]{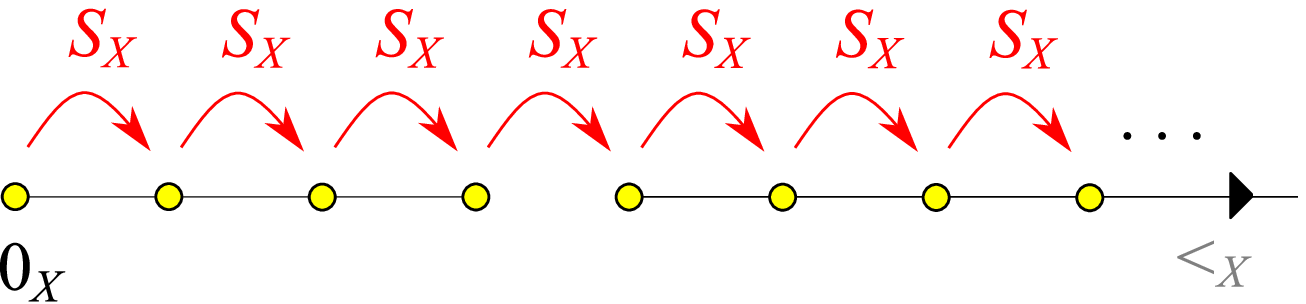}
  \caption{Counterexample for implication (I-18) in Fig.~\ref{F:ArrowDiagramMod1}.}\label{F:NincompleteOrder}
\end{figure}

\begin{figure}[ht]
  \centering
  \includegraphics[width = 12cm]{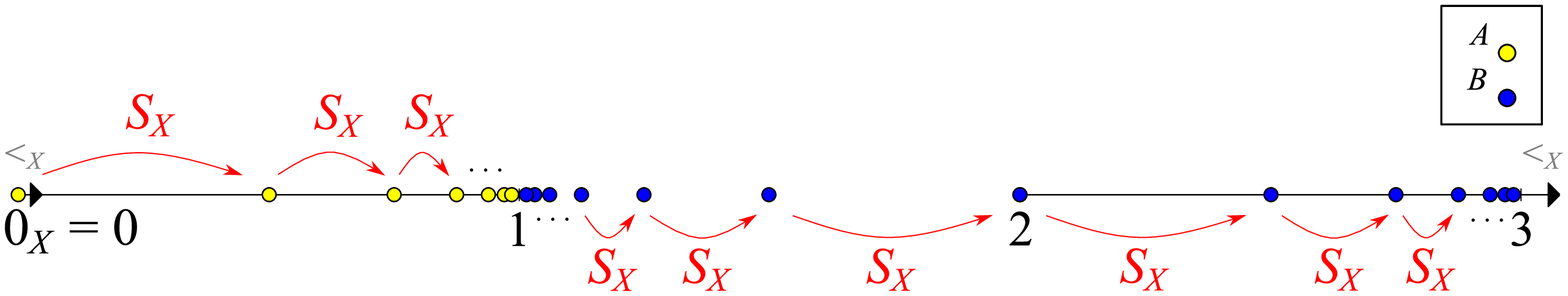}
  \caption{Counterexample for implication (I-9) in Fig.~\ref{F:ArrowDiagramMod1}.}\label{F:omegaplusZmod}
\end{figure}

\section{Conclusion} Although it is often said otherwise, complete induction and well-ordering are not equivalent to regular (ordinary) induction. The equivalence holds only if we assume, additionally, that every nonzero natural number is a successor. On the other hand, strong induction, as defined in this paper, is indeed equivalent to regular induction. We hope this work can shed light on the need for more caution when alleging equivalence between induction principles for $\mathbb{N}$.

\end{document}